\documentclass[11pt]{article}

\usepackage{graphicx}
\usepackage{amsmath}
\usepackage{amsthm}
\usepackage{amsfonts}
\usepackage{amssymb, amscd}

\usepackage[all, knot]{xy}

\newtheorem{thm}{Theorem}[section]

\newtheorem{lem}[thm]{Lemma}
\newtheorem{prop}[thm]{Proposition}
\newtheorem{example}[thm]{Example}
\theoremstyle{definition}
\newtheorem{defn}[thm]{Definition}
\theoremstyle{remark}
\newtheorem{rem}[thm]{Remark}
\numberwithin{equation}{section}

\input{epsf.sty}

\begin{document}

\title{Cohomology of $n$-ary-Nambu-Lie  superalgebras and \\ super $w_{\infty}$ 3-algebra}

\author{Faouzi AMMAR \\\small{ Universit\'{e} de Sfax}\\ \small{Facult\'{e} des Sciences,}\\
\small{B.P. 1171, Sfax 3000, Tunisie,}\\ \small{Faouzi.Ammar@rnn.fss.tn}
\and Nejib Saadaoui\\\small{ Universit\'{e} de Gab\`{e}s}\\ \small{Facult\'{e} des Sciences,}\\
\small{Campus Universitaire, Gab\`{e}s 6072, Tunisie,}\\ \small{Nejib.Saadaoui@fsg.rnu.tn}}

\maketitle

\begin{abstract}
The purpose of this paper is define the representation and the cohomology of n-ary-Nambu-Lie superalgebras. Morever we study central extensions and provide as application the computations of the derivations and second cohomology group of super $w_{\infty}$ 3-algebra.
\end{abstract}



\section*{Introduction}

The first instances of n-ary algebras in Physics appeared with a generalization of the Hamiltonian mechanics
proposed in 1973 by Nambu \cite{25S}. More recent motivation comes from string theory and M-branes
involving naturally an algebra with ternary operation called Bagger-Lambert algebra which gives impulse
to a significant development. It was used in \cite{8S} as one of the main ingredients in the construction of a new
type of supersymmetric gauge theory that is consistent with all the symmetries expected of a multiple
M2-brane theory: 16 supersymmetries, conformal invariance, and an SO(8) R-symmetry that acts on the
eight transverse scalars. On the other hand in the study of supergravity solutions describing M2-branes
ending on M5-branes, the Lie algebra appearing in the original Nahm equations has to be replaced with
a generalization involving ternary bracket in the lifted Nahm equations, see \cite{9Sami}. For other applications in
Physics see \cite{27Sami,28Sami,29Sami}.\\
The algebraic formulation of Nambu mechanics is due to Takhtajan \cite{12Sami,31Sami} while the abstract definition of
n-ary Nambu algebras or n-ary Nambu-Lie algebras (when the bracket is skew symmetric) was given by
Filippov in 1985 see \cite{14Sami}. The Leibniz n-ary algebras were introduced and studied in \cite{11Sami}. For deformation
theory and cohomologies of n-ary algebras of Lie type, we refer to \cite{4Sami,5Sami,15Sami,13Sami,31Sami}.\\

The paper is organized as follows. In the first section we give the definitions and some key constructions of n-ary-Nambu-Lie superalgebras. In Section $2$ we define a derivation of n-ary-Nambu-Lie superalgebra. Section $3$ is dedicated to the representation theory
 n-ary-Nambu-Lie superalgebras including adjoint  representation. In Section $3$ we construct   family of  cohomologies of n-ary-Nambu-Lie superalgebras. In Section $4$, we discuss extensions of n-ary-Nambu-Lie superalgebras and their connection to cohomology.  In the last section we compute the derivations  and  cohomology group of the super $w_{\infty}$ 3-algebra.

\section{The $n$-ary-Nambu-Lie superalgebra}
Let  $\mathcal{N}$ be a linear superspace over a field $\mathbb{K}$ that is a $\mathbb{Z}_{2}$-graded linear space with a direct sum $\mathcal{N}=\mathcal{N}_{0}\oplus \mathcal{N}_{1}.$\\
The elements of $\mathcal{N}_{j}$, $j\in \mathbb{Z}_{2},$ are said to be homogenous  of parity $j.$ The parity of  a homogeneous element $x$ is denoted by $|x|.$\\
Let $x_{1},\cdots,x_{n}$ be $n$ homogenous  elements of $\mathcal{N}$, we denote by $|(x_{1},\cdots ,x_{n})|=|x_{1}|+\cdots+|x_{n}|$ the parity of an element $(x_{1},\cdots,x_{n})$ in $\mathcal{N}^{n}$.\\
The space $End (\mathcal{N})$ is $\mathbb{Z}_{2}$ graded with a direct sum $End (\mathcal{N})=(End (\mathcal{N}))_{0}\oplus(End (\mathcal{N}))_{1}$ where
$(End (\mathcal{\mathcal{N}}))_{j}=\{f\in End (\mathcal{\mathcal{N}}) /f (\mathcal{\mathcal{N}}_{i})\subset \mathcal{\mathcal{N}}_{i+j}\}.$
The elements of $(End (\mathcal{\mathcal{N}}))_{j}$  are said to be homogenous of parity $j.$

\begin{defn} \cite{Flipov}
An $n$-ary-Nambu  superalgebra is a pair $(\mathcal{N},[.,\cdots,.])$ constiting of a vector superspace $\mathcal{N}$ and  even $n$-linear map $[.,\dots,.]:\mathcal{N}^n\rightarrow \mathcal{N}$, satisfying
\begin{eqnarray}
\Big[y_{2},\dots,y_{n},[x_{1},\dots ,x_{n}]\Big]&=&\displaystyle \sum_{i=1}^{n} (-1)^{(|y_{2}|+\dots +|y_{n}|)(|x_{1}|+\dots + |x_{i-1}|)} \Big[x_{1},\dots,
 [y_{2},\dots,y_{n},x_{i}],\dots,x_{n}\Big]\nonumber\\ \label{nambu}
 \end{eqnarray}
 for all $( x_{1},\cdots ,x_{n}) \ \in \mathcal{N}^n$, $(y_{2},\cdots,y_{n} )\in \mathcal{N}^{n-1}$.\\
  The identity (\ref{nambu}) is called Nambu identity.
\end{defn}
\begin{defn}
An $n$-ary-Nambu  superalgebra $(\mathcal{N},[.,\cdots,.])$ is called $n$-ary-Nambu-Lie superalgebra if the bracket is skew-symmetric that is
\begin{eqnarray}
[x_{1},\dots,x_{i-1},x_{i},\dots,x_{n}]&=&-(-1)^{|x_{i-1}||x_{i}|}[x_{1},\dots,x_{i},x_{i-1},\dots,x_{n}].\label{skew}
\end{eqnarray}
\end{defn}
\begin{defn}
Let $(\mathcal{N},[.,\dots,.])$    and         $(\mathcal{N}',[.,\dots,.]')$  be two $n$-ary-Nambu-Lie superalgebra. An homomorphism $f:\mathcal{N}\rightarrow \mathcal{N}'$ is said to be morphism of $n$-ary-Nambu-Lie superalgebra if
\begin{eqnarray}
[f(x_{1}),\dots,f(x_{n})]'&=&f([x_{1},\dots,x_{n}])\ \forall x_{1},\dots,x_{n}\in \mathcal{N}.
\end{eqnarray}
\end{defn}
\begin{prop}
Let $f$ be an even endomorphism of $n$-ary-Nambu-Lie superalgebra    $(\mathcal{N},[.,\cdots,.])$. We can define the new  $n$-ary-Nambu-Lie superalgebra    $(\mathcal{N},[.,\cdots,.]')$ , where $[x_{1},\dots,x_{n}]'=f([x_{1},\cdots,x_{n}])$.
\end{prop}
\section{Derivation of $n$-ary-Nambu-Lie superalgebra}
\begin{defn}\cite{Flipov}
We call  $D\in End_{i}(\mathcal{N})$, where  $i$ is in $\mathbb{Z}_{2}$, a  derivation of the $n$-ary-Nambu-Lie superalgebra $(\mathcal{N},[.,\cdots, .])$ if
\begin{eqnarray*}
D([x_{1},\cdots,x_{n}])&=&\sum_{k=1}^{n}(-1)^{|D|(|x_{1}|+\cdots +|x_{k-1}|)}[x_{1},\cdots,D(x_{k}),\cdots,x_{n}],\ \ for \ all\  homogeneous \ x_{1},\cdots, x_{n} \in \mathcal{N}.
\end{eqnarray*}
\end{defn}
We denote by   $Der(\mathcal{N})=Der_{\overline{0}}(\mathcal{N})\oplus Der_{\overline{1}}(\mathcal{N})$   the set of derivation of the $n$-ary-Nambu-Lie superalgebra $\mathcal{N} $.\\

The subspace $Der (\mathcal{N})\subset End (\mathcal{N})$ is easily seen to be closed under the bracket
\begin{eqnarray}
[D_{1},D_{2}]=D_{1}\circ D_{2}-(-1)^{|D_{1}||D_{2}|}D_{2}\circ D_{1} \label{crochet}
\end{eqnarray}
(known as the supercommutator) and it is called the superalgebra of derivations of $\mathcal{N}$.\\
With above notation, $Der(\mathcal{N})$ is a Lie superalgebra, in wich the Lie bracket is given by (\ref{crochet}).\\
Fix $n-1$ homogeneous elements $x_{1},\cdots,x_{n-1}\in \mathcal{N}$, and define the transformations $ad(x_{1},\cdots,x_{n-1})\in End(\mathcal{N})$ by the rule
\begin{eqnarray}
ad(x_{1},\cdots,x_{n-1})(x)=[x_{1},\cdots,x_{n-1},x]. \label{ad}
\end{eqnarray}
Then $ad(x_{1},\cdots,x_{n-1})$ is a derivation of $\mathcal{N}$, wich we call inner derivation of $\mathcal{N}$.\\
Indeed we have
\begin{eqnarray*}
 ad(y_{2},\cdots,y_{n})\Big( [x_{1},\cdots,x_{n}]\Big) &=& \Big[y_{2},\cdots,y_{n},[x_{1},\cdots,x_{n}]\Big]\\
 &=&\sum_{i=1}^{n} (-1)^{(|y_{2}|+\cdots +|y_{n}|)(|x_{1}|+\cdots + |x_{i-1}|)} \Big[x_{1},\cdots,
ad(y)(x_{i}),\cdots,x_{n}\Big]\\
 &=&\sum_{i=1}^{n} (-1)^{|ad(y)|(|x_{1}|+\cdots + |x_{i-1}|)} \Big[x_{1},\cdots,
ad(y)(x_{i}),\cdots,x_{n}\Big].\\
\end{eqnarray*}


\section{Representations of $n$-ary-Nambu-Lie superalgebra}
We provide in the following a graded version of the study of representations of $n$-ary-Nambu-Lie algebra stated in \cite{Exten}.\\

Let $(\mathcal{N},[.,\cdots, .])$ be a  $n$-ary-Nambu-Lie superalgebra and $V=V_{\overline{0}}\oplus V_{\overline{1}} $
an arbitrary vector superspace. Let
$[.,.]_{V}:\mathcal{N}^{n-1}\times V \longrightarrow V$ be a bilinear map satisfying $[\mathcal{N}^{n-1}_{i},V_j]_V\subset V_{i+j}$
where $i,\ j \in \mathbb{Z}_{2}$.
\begin{defn}
The pair $(V,[.,.]_{V} )$ is called a module on the $n$-ary-Nambu-Lie  superalgebra $\mathcal{N}=\mathcal{N}_{\overline{0}}\oplus \mathcal{N}_{\overline{1}}$
or $\mathcal{N}$-module $V$ if the even multilinear mapping $[.,\dots,.]_V$ satisfies

\begin{eqnarray}
 &&\Big[ ad(x)(x_{n}),y_{2},\dots, y_{n-1},v\Big]_{V}\nonumber\\
 &=& \sum_{i=1}^{n}(-1)^{n-i+|x_{i}|(|x_{i+1}|+\dots+|x_{n}|)}
 \Big[x_{1}, \dots , \widehat{x_{i}} \dots , x_{n},\big[x_{i},y_{2}, \dots , y_{n-1},v\big]_{V}\Big]_{V},\label{rep1}\\
&&\sum_{i=1}^{n-1}(-1)^{|y|(|x_{1}|+\dots + |x_{i-1}|)} \Big[x_{1},\dots,ad(y)(x_{i}),\dots,x_{n-1},v\Big]_V=\Big[y,[x,v]_V\big]_V-(-1)^{|y||x|}\big[x,[y,v]_V\Big]_V\nonumber\\\label{rep2}
\end{eqnarray}
for all homogeneous $x,\ y$ in $\mathcal{N}^{n-1}$ and $v$ in $V$.
It will also say that $(V,[.,\dots,.]_{V} )$ is a representation of $\mathcal{N}$.
\end{defn}
\begin{example}
Let $(\mathcal{N},[.,\dots, .])$ be a  $n$-ary-Nambu-Lie superalgebra and the map $ad$  defined in (\ref{ad}). Then $(\mathcal{N},ad)$ is  a representation of $\mathcal{N}$.
\end{example}


\begin{rem}
When $[.,.]_{V}$ is the zero map, we say that the module $V$ is trivial.
\end{rem}
\section{ Cohomology of n-ary-Nambu-Lie superalgebra induced by cohomology of Leibniz  algebras}
In this section, we aim to extend to n-ary Nambu superalgebra  type of process introduced by
Takhtajan to construct a complex of  n-ary Nambu superalgebra   starting from a complex of binary algebras (see \cite{Lmars2013}).\\
\begin{defn}
A  Leibniz superalgebra is a pair $(L,[.,.])$ consisting of a vector superspace $L$ and bilinear map $[.,.]:L\times L\rightarrow L$ satisfying
\begin{equation}\label{lei}
    \Big[x,[y,z]\Big]=  \Big[[x,y],z\big]+(-1)^{|x||y|}   \big[y,[x,z]\Big].
\end{equation}
Let  $(L,[.,.])$ be a Leibniz superalgebra and $W$ be an arbitrary vector superspace. Let $[.,.]_W:
L\times W\longrightarrow W$ be a even bilinear map  satisfiying $$ \Big[[x,y],w\big]_W= \Big[x,[y,w]_W\Big]_W-(-1)^{|x||y|}   \big[y,[x,v]_w\Big]_W.$$
The pair $(W,[.,.]_W$ is called an $L$-module.\\
\end{defn}

Let $(\mathcal{N},[.,\dots,.])$   be a  $n$-ary-Nambu-Lie superalgebra and  $(V,[.,\dots,.]_{V} )$
  be a  $\mathcal{N}$-module.\\
    We denote by $\mathcal{L}(\mathcal{N})$ the space $\wedge^{n-1}\mathcal{N}$ and we call it the fundamental set.\\
  We define a bilinear map $[.,.]_{L}:\mathcal{L}(\mathcal{N})\times \mathcal{L}(\mathcal{N})  \rightarrow \mathcal{L}(\mathcal{N})$  and $ad:\mathcal{N}\times  V\longrightarrow V$ respectivly  by
\begin{eqnarray}
[x,y]_{L}&=&\displaystyle \sum_{i=1}^{n-1}(-1)^{|x|(|y_{1}|+\cdots +|y_{i-1}|)}y_{1}\wedge\dots\wedge ad(x)(y_{i})\wedge\dots \wedge y_{n-1}\label{ladd}
\end{eqnarray}
and
\begin{eqnarray}
ad(x,v)&=&[x,v]_{V}\label{add}
\end{eqnarray}
for all $x=x_{1}\wedge \dots \wedge x_{n-1},\ y=y_{1}\wedge \dots \wedge y_{n-1}\in \mathcal{L}(\mathcal{N}),\ v\in V.$

\begin{lem}\label{30mars}
Let $(\mathcal{N},[.,\dots,.])$   be a  $n$-ary-Nambu-Lie superalgebra and be
$V$  be a  $\mathcal{N}$-module.
The map $ad$  satisfies
\begin{equation}\label{29sep}
   ad([x,y]_{L})(v)= ad(x)(ad(y)(v))-(-1)^{|x||y|}ad(y)(ad(x)(v))
\end{equation}
for all $x,\ y \ \in \mathcal{L}(\mathcal{N}),\ v\in V$.
\end{lem}
\begin{proof}
By (\ref{add}) and (\ref{rep2}) we have
\begin{eqnarray*}
  ad([x,y]_{L})(v)
    &=&\big[[x,y]_{L},v\big]_V\\
  &=&\sum_{i=1}^{n-1}(-1)^{|x|(|y_{1}|+\dots +|y_{i-1}|)}[ y_{1}\wedge\dots\wedge ad(x)(y_{i})\wedge\dots \wedge y_{n-1},v]_V.\\
   &=&\big[x_{1},\dots,x_{n-1},[y_{1},\dots,y_{n-1},v]\big]_V -  (-1)^{|x|(|y_{1}|+\cdots +|y_{n-1}|)}\big[y_{1},\dots,y_{n-1},[x_{1},\dots,x_{n-1},v]\big]_V\\
   &=&ad(x)(ad(y)(v)) -  (-1)^{|x||y|}ad(y)(ad(x)(v)).
\end{eqnarray*}
\end{proof}

\begin{prop}
The pair $(\mathcal{L}(\mathcal{N}),[.,.]_{L})$ is a  Leibniz superalgebra.\\
\end{prop}
Let  $(V,[.,\dots,.]_{V})$ be a $\mathcal{N}$-module. We denote
$$W= \mathcal{L}(\mathcal{N},V)=\{x_1\wedge \dots \wedge x_{n-2}\wedge v, x_i\in \mathcal{N},v\in V\}.$$
  \textrm{Note that}
\begin{eqnarray} \label{symetrique}
 &&  u_{1}\wedge \dots u_{j-1}\wedge u_{j} \wedge \dots \wedge u_{n-1} =-(-1)^{|u_{j-1}||u_{j}|}u_{1}\wedge \dots u_{i}\wedge u_{i-1} \wedge \dots \wedge u_{n-1},
\end{eqnarray}
for all  homogenous element $u=u_{1}\wedge \dots \wedge u_{n-1}$ of  \  $W$.\\
Define a bilinear map $[.,.]_W:\mathcal{L}(\mathcal{N})\longrightarrow W$ by
\begin{eqnarray*}
[x,y_{1}\wedge \dots \wedge y_{1}\wedge v]_W&=&\sum_{i=1}^{n-1}(-1)^{|x|(|y_{1}|+\dots +|y_{i-1}|)}y_{1}\wedge\dots\wedge ad(x)(y_{i})\wedge\dots \wedge y_{n-1} \wedge v \\
&&+(-1)^{|x|(|y_{1}|+\dots +|y_{n-1}|)}y_{1}\wedge\dots \wedge y_{n-1}\wedge ad(x)(v).
\end{eqnarray*}
\begin{prop}
The pair $(W,[.,.]_{W})$ is a  $L$-module.\\
\end{prop}

In the following, the expression $[x,y]$ means:
\begin{itemize}
\item $[x,y]_L$ if $x,\ y\in \mathcal{L}(\mathcal{N})$.
\item $[x_1,\dots,x_{n-1},y]$ if $x=x_1\wedge \dots \wedge x_{n-1} \in \mathcal{L}(\mathcal{N}),$ $y\in \mathcal{N}$.\\
\end{itemize}

\begin{defn}
We call $k$-cochain of a $n$-ary super-algebra $\mathcal{N}$ with values in $V$ a multilinear map
$$\varphi:\mathcal{L}(\mathcal{N})^{k}\times \mathcal{N} \longrightarrow V.$$
Denote   $C^{k}(\mathcal{N}, V)$ the set of   $k$-cochains on $\mathcal{N}$ with values in $V$.\\
\end{defn}
\begin{thm}
 Let $(\mathcal{N},[.,\cdots, .])$ be a  $n$-ary-Nambu-Lie superalgebra and  $ C^{k}(\mathcal{L}(\mathcal{N}),V)$ the set of   $k$-cochain.\\
We define a coboundary operator  $d^{k}:C^{k}(\mathcal{L}(\mathcal{N}),W)\rightarrow C^{k+1}(\mathcal{L}(\mathcal{N}),W)$ by $df(x)=-[x,f]$ when $f\in C^{0}(\mathcal{L}(\mathcal{N}),V)=V$ and for $k\geq 1,$

\begin{eqnarray*}
  d^{k} (f)(x_{0},\dots,x_{k})
&=& -\sum_{0\leq s < t\leq k}(-1)^{s+|x_{s}|(|x_{s+1}|+\dots+|x_{t-1}|)}           f (x_{0},\dots,\widehat{x_{s}},\dots,x_{t-1},[x_{s},x_{t}],x_{t+1},\dots,x_{k},z)\ \ \ \ \  \\
&&+\sum_{s=0}^{k-1}(-1)^{s+|x_{s}|(|f|+|x_{0}|+\dots+|x_{s-1}|)}\Big[x_{s},f (x_{0},\dots,\widehat{x_{s}},\dots,x_{k})\Big]_{W} \\
&&+(-1)^{k}
                     \Big[ f(x_{0},\dots,x_{k-1}),x_k      \Big]_{W}'  ,
\end{eqnarray*}
where
\begin{eqnarray*}
&&\Big[ x_{1}\wedge \dots \wedge x_{n-2}\wedge v,y_{1} \wedge \dots \wedge y_{n- 1}    \Big]_{W}'\\
&=&-\sum_{i=1}^{n-1}(-1)^{(|x_1|+\dots+|x_{n-2}|+|v|)(|y_1|+\dots+|y_{i-1}|)+|v||y_i|}y_1\wedge \dots \wedge ad(x_{1}\wedge \dots \wedge x_{n-2}\wedge y_i)(v)\wedge \dots y_{n-1}
\end{eqnarray*}
Let $\Delta^{k}:C^{k-1}(\mathcal{N},V)\rightarrow C^{k}(\mathcal{L}(\mathcal{N}), W)$ be the linear map defined for $k=0$ by
$$\Delta (f)(x_0)=\sum_{i=1} ^{n-1}
(-1)^{|f|(|x_{0}^{1}|+\dots+|x_{0}^{i-1}|)}x_0^{1}\wedge \dots \wedge f(x_0^{i})\wedge\dots \wedge  x_{0}^{n-1}$$
and for $k>0$ by
\begin{eqnarray*}
&&\Delta^{k} (f)(x_0,\dots,x_k)\\
&=& \sum_{i=1} ^{n-1}
(-1)^{(|f|+|x_0|+\dots+|x_{k-1}|)(|x_{k}^{1}|+\dots+|x_{k}^{i-1}|)}
 x_{k}^{1}\wedge \dots \wedge f( x_0,\dots , x_{k-1}, x_{k}^{i}) \wedge x_{k}^{i+1}\wedge \dots x_{k}^{n-1}
\end{eqnarray*}
where we set $x_j=x_{j}^{1}\wedge \dots \wedge x_{j}^{n-1}$.
Then there exists a cohomology complex $(C^{k}(\mathcal{N},\ \mathcal{N}),\delta )$ for $n$-ary-Nambu-Lie superalgebra such that
$$ d^{k}\circ \Delta ^{k-1}=\Delta ^{k-1}\circ \delta ^{k-1}  $$
The cobondary map $\delta^{k+1}:\mathcal{C}^{k}(\mathcal{N},V)\rightarrow \mathcal{C}^{k+1}(\mathcal{N},V)$  is defined by
\begin{eqnarray*}
 && \delta^{k+1}   (f)(x_{0},\dots,x_{k},z) \nonumber \\
&=& -\sum_{0\leq s < t\leq k}(-1)^{s+|x_{s}|(|x_{s+1}|+\dots+|x_{t-1}|)}f (x_{0},\dots,\widehat{x_{s}},\dots,x_{t-1},[x_{s},x_{t}],x_{t+1},\dots,x_{k},z)\ \ \ \ \ \ \   \label{cob1} \\
&&-\sum_{s=0}^{k}(-1)^{s+|x_{s}|(|x_{s+1}|+\dots+|x_{k}|)}
f \Big(x_{0},\dots,x_{s-1},x_{s+1},\dots,x_{k},ad(x_{s})(z)\Big)  \label{cob1g} \\
&&+\sum_{s=0}^{k}(-1)^{s+|x_{s}|(|f|+|x_{0}|+\dots+|x_{s-1}|)}\Big[x_{s},f (x_{0},\dots,\widehat{x_{s}},\dots,x_{k},z)\Big]_{V}.\label{cob2}  \\
&&+\sum_{i=1}^{n-1}(-1)^{k-i+(|f|+|x_{0}|+\dots+|x_{k-1}|+|x_{k}^{i+1}|+\dots+ |x_{k}^{n-1}|)(|x_{k}^{i}|+|x_{k}|)+|z|(|f|+|x_{0}|+\dots+|x_{k}|)
+|x_{k}|(|x_{k}^{i+1}|+\dots+|x_{k}^{n-1}|)}\nonumber\\
                      && \Bigg[z\wedge x^{1}_{k}\wedge \dots \widehat{ x^{i}_{k}}\wedge  \dots \wedge x^{n-1}_{k} ,
    f \Big(x_{0},\dots,x_{k-1},x_{k}^{i}\Big)\Bigg]_{V}  .\label{def cob6}
\end{eqnarray*}
\end{thm}
\begin{proof}
For any$f\in C^{k-2}(\mathcal{N},V)$, by calculation we obtain $d^{k+1}\circ d^{k}(f)(x_0,\dots,x_k)=0$ and \\
$d^{k}\circ \Delta ^{k-1}(f)(x_0,\dots,x_k)=\Delta ^{k-1}\circ \delta ^{k-1} (f)(x_0,\dots,x_k)$.\\

One has $\Delta ^{k+1}\circ \delta ^{k}=d^{k}\circ \Delta ^{k},$ then $\Delta ^{k+1} \circ \delta ^{k} \circ \delta ^{k-1}= d^{k}\circ d^{k-1}\circ \Delta ^{k-1}=0$, because $d^{k}\circ d^{k-1}=0$
\end{proof}

\begin{defn}
 \begin{itemize}
\item The  $k$-cocycles space is defined as $Z^{k}(\mathcal{N},V)=\ker \ \delta^{k}$. The  even (resp. odd)  $k$-cocycles space  is defined as
$Z^{k}(\mathcal{N},V)_{0}=Z^{k}(\mathcal{N},V)\cap (C^{k}(\mathcal{N},\ V))_0$ (resp. $Z^{k}(\mathcal{N},V)_{1}=Z^{k}(\mathcal{N},V)
\cap (C^{k}(\mathcal{N},\ V))_1$).
 \item The $k$-coboundaries space is defined as  $B^{k}(\mathcal{N},V)=Im \ \delta^{k-1}$. The even (resp. odd)  $k$-coboundaries space is $B_{0}^{k}(\mathcal{N},V)=B^{k}(\mathcal{N},V)\cap (C^{k}(\mathcal{N},\ V))_0$ (resp. $B_{1}^{k}(\mathcal{N},V)=B^{k}(\mathcal{N},V)\cap (C^{k}(\mathcal{N},\ V))_0$).
\item The $k^{th}$  cohomology space is the quotient $H^{k}(\mathcal{N},V)= Z^{k}(\mathcal{N},V)/ B^{k}(\mathcal{N},V)$. It decomposes as well as   even and odd  $k^{th}$  cohomology spaces.
  \end{itemize}
 Finally, we denote by $H^{k}(\mathcal{N},V)=H_{0}^{k}(\mathcal{N},V) \oplus H_{1}^{k}(\mathcal{N},V)$ the $k^{th}$ cohomology space and by $\oplus_{k\geq 0}H^{k}(\mathcal{N},V)$ the $r$-cohomology group of the Hom-Lie superalgebra $\mathcal{N}$ with values in $V$.
 \end{defn}
\begin{rem}
The subspace $Z^{1}(\mathcal{N},\mathcal{N})$ is the set of derivation of $\mathcal{N}$.
\end{rem}
\section{Extensions of n-ary-Nambu-Lie superalgebra }
An extension theory of  Hom-Lie superalgebras was stated  in \cite{Exten Hom}.\\

An extension of a n-ary-Nambu-Lie superalgebra $(\mathcal{N},[.,\dots,.])$ by $\mathcal{N}$-module  $(V,[.,\dots,.]_{V})$ is an exact sequence
$$0\longrightarrow (V,[.,\dots,.]_{V})\stackrel{i}{\longrightarrow} (\widetilde{\mathcal{N}},\widetilde{[.,\dots,.]})\stackrel{\pi}{\longrightarrow }(\mathcal{N},[.,\dots,.]) \longrightarrow 0 .$$
We say that the extension is central if $[ \mathcal{L}(\widetilde{\mathcal{N}}), i(V)]_{\widetilde{\mathcal{N}}}=0.$\\
Two extensions
$$0 \longrightarrow  (V,[.,\dots,.]_{V})\stackrel{i_{k}}{\longrightarrow}     (\widetilde{\mathcal{N}_{k}},\widetilde{[.,\dots,.]})\stackrel{\pi_{k}}{\longrightarrow } (\mathcal{N},[.,\dots,.]) \longrightarrow 0 \ \ \ (k=1,2)$$
are equivalent if there is an isomorpism $\varphi:(\mathcal{N}_{1},[.,\dots,.]_{1})\longrightarrow (\mathcal{N}_{2},[.,\dots,.]_{2})$ such that $\varphi o \ i_{1}= i_{2}$ and $\pi_{2}\ o\ \varphi=\pi_{1}.$\\

\begin{prop}
Let $(\mathcal{N},[.,\cdots, .])$ be a  $n$-ary-Nambu-Lie superalgebra
and $V$ be a $\mathcal{N}$-module. The second cohomology space $H^{2}(\mathcal{N},V)=Z^{2}(\mathcal{N},V)/ B^{2}(\mathcal{G},V)$ is  in one-to-one correspondence with the set of the equivalence classes of  central extensions of $(\mathcal{N},[.,\dots, .])$ by $(V,[.,\dots,.]_{V}).$
\end{prop}
\begin{proof}
Let $$0\longrightarrow (V,[.,\dots,.]_{V})\stackrel{i}{\longrightarrow} (\widetilde{\mathcal{N}},\widetilde{[.,\dots,.]})\stackrel{\pi}{\longrightarrow }(\mathcal{N},[.,\dots,.]) \longrightarrow 0 .$$ be  a central extension of $n$-ary-Nambu-Lie superalgebra   $(\mathcal{N},[.,\dots, .])$    by  $(V,[.,\dots,.]_{V}),$
so there is a  space $H$ such that $\widetilde{\mathcal{N}}=H \oplus i(V).$

The map  $\pi_{/H}:H\rightarrow \mathcal{N}$ (resp $k:V\rightarrow i(V) $) defined by $\pi_{/H}(x)=\pi(x) $ (resp. $k(v)=i(v)$) is bijective, its inverse s (resp. $l$) note.  Considering the map $\varphi:\mathcal{N}\times V \rightarrow \widetilde{\mathcal{N}}$  defined  by
  $ \varphi(x,v)= s(x)+i (  v),$    it is easy to verify that $\varphi$ is a bijective.\\
  For all $x=x_{1}\wedge \dots x_{n-1}\in \mathcal{L}(\mathcal{N})$  (resp. $v=v_{1}\wedge\dots v_{n-1}\in \mathcal{L}(V)$) we denote $s(x_{1})\wedge \dots s(x_{n-1})$ (resp. $i(v_{1})\wedge \dots \wedge i(v_{n-1})$) by $s(x).$ (resp. $i(v)$).\\
  Since $\pi$ is homomorphism of    $n$-ary-Nambu-Lie superalgebra  then $\pi \Big([s(x),s(z)]_{\widetilde{\mathcal{N}}}-s([x,z])\Big)=0$\\
  so $[s(x),s(z)]_{\widetilde{\mathcal{N}}}-s([x,z])\in i(V).$\\
We set $[s(x),s(z)]-s([x,z])=G(x,z)\in i(V)$  then $F(x,z)=l\circ G(x,z) \in V,$  it easy to see  that $F \in C^2(\mathcal{N},V)$ is a $2$-cochain that defines a bracket on $\widetilde{\mathcal{N}}.$ In fact, we can identify as a superspace $\mathcal{N}\times V$ and $\widetilde{\mathcal{N}}$ by $\varphi :(x,v)\rightarrow s(x)+i(v)$ where the bracket is
 $$[s(x)+i(v),s(z)+i(w)]_{\widetilde{\mathcal{N}}}=[s(x),s(z)]_{\widetilde{\mathcal{N}}}=s([x,z])+F(x,z).$$
   Viewed  as elements of $\mathcal{N}\times V$ we have   $\Big[(x,v),(z,w)\Big]=\Big([x,z],F(x,z)\Big)$ and the homogeneous elements $(x,v)$ of $\mathcal{N}\times V$ are such that $|x|=|v|$ and we have in this case $|(x,v)|=|x|$.\\
We deduce that for every central extension
$$0\longrightarrow (V,[.,\dots,.]_{V})\stackrel{i}{\longrightarrow} (\widetilde{\mathcal{N}},\widetilde{[.,\dots,.]})\stackrel{\pi}{\longrightarrow }(\mathcal{N},[.,\dots,.]) \longrightarrow 0 .$$
  One may associate a two cocycle $F\in Z^{2}(\mathcal{N},V)$. Indeed, for $x \in \mathcal{L}(\mathcal{N}),z \in \mathcal{N} ,$ if we set $$F(x,z)=l\Big([s(x),s(z)]-s([x,z])\Big)\in V,$$
then, we have $F(x,z)\in V$ and $F$ satisfies the $2$-cocycle conditions.\\

Conversely, for each $f\in Z^{2}(\mathcal{N},V),$  one can define a central extension
$$0\longrightarrow (V,[.,\dots,.]_{V})\longrightarrow (\mathcal{N}_{f},[.,\dots,.]_{f})\longrightarrow (\mathcal{N},[.,\dots,.]) \longrightarrow 0 ,$$
by $$\Big[(x,v),(y,w)\Big]_{f}=\Big([x,y],f(x,y)\Big),$$ where $x\in \mathcal{L}(\mathcal{N}), \ z \in \mathcal{N}$ and $v\in \mathcal{L}(V),\ w \in V.$\\
Let $f$ and $g$ be two elements of $ Z^{2}(\mathcal{N},V)$ such that $f-g \in B^{2}(\mathcal{G},V)$ i.e. $(f-g)(x,z)=h([x,z]),$ where $h:\mathcal{N}\rightarrow V $ is  a linear map . Now we prove that the extensions  defined by $f$ and $g$ are equivalent. Let us define
 $\Phi:\mathcal{N}_{f}\times V \rightarrow \mathcal{N}_{g}\times V$ by
  $$ \Phi (x,v)=(x,v-h(x)). $$
It is clear that $\Phi$ is bijective. Let us check that $\Phi$ is a homomorphism of n-ary-Nambu-Lie superalgebra . We have
\begin{eqnarray*}
  \Big[\Phi((x,v)),\Phi((z,w))\Big]_{g}&=&\Big[(x,v-h(x)),(z,w-h(z))\Big]_{g}\\
  &=&\Big([x,z],g(x,z)\Big)\\
  &=&\Big([x,z],f(x,z)-h([x,z])\Big)\\
  &=&\Phi\Big(([x,z],f(x,z))\Big)\\
  &=&\Phi\Big([(x,v),(z,w)]_{f}\Big).
\end{eqnarray*}
Next, we show that for $f,g \in Z^{2}(\mathcal{N},V)$ such that the central extensions\\
 $0\rightarrow (V,[.,\dots,.]_{V})\rightarrow (\mathcal{N}_{f},\widetilde{[.,\dots,.]}_{f} )\rightarrow (\mathcal{N}, [.,\dots,.]    ) \rightarrow 0 ,$ and

 $0\rightarrow (V,[.,\dots,.]_{V})\rightarrow (\mathcal{N}_{g},\widetilde{[.,\dots,.]}_{g} )\rightarrow (\mathcal{N}, [.,\dots,.]    ) \rightarrow 0 ,$
are equivalent, we have $f-g\in B^{2}(\mathcal{N},V).$ Let $\Phi$ be a homomorphism of n-ary-Nambu-Lie superalgebra .  such that
\begin{displaymath}
\xymatrix { 0 \ar[r] & (V,[.,\dots,.]_{V}) \ar[d]_{id_{V}}\ar[r]^{i_{1}} & (\mathcal{N}_{f},\widetilde{[.,\dots,.]} ) \ar[d]_{\Phi} \ar[r]^{\pi_{1}}&(\mathcal{N},[.,\dots,.] )\ar[r] \ar[d]_{id_{\mathcal{N}}}&0 \\
0\ar[r] & (V,[.,\dots,.]_{V}) \ar[r]^{i_{2}}     & (\mathcal{N}_{g},\widetilde{[.,\dots,.]}         ) \ar[r]^{\pi_{2}} &(\mathcal{N},[.,\dots,.])\ar[r]&0}
\end{displaymath}
commutes. We can express $\Phi(x,v)=(x,v-h(x))$ for some linear map $h:\mathcal{N}\rightarrow V.$ Then we have
\begin{eqnarray*}
\Phi([(x,v),(z,w)]_{f})&=&\Phi(([x,z],f(x,z)))\\
&=&([x,z],f(x,z)-h([x,z])),
\end{eqnarray*}
\begin{eqnarray*}
[\Phi((x,v)),\Phi((z,w))]_{g}&=&[(x,v-h(x)),(z,w-h(y))]_{g}\\
&=&([x,z],g(x,z)),
\end{eqnarray*}
and thus $(f-g)(x,y)=h([x,z])$ (i.e. $f-g \in B^{2}(\mathcal{N},V)$), so we have completed the proof.
\end{proof}

 \subsection{ Deformation of      n-ary-Nambu-Lie superalgebra.}
 \begin{defn}
Let $( \mathcal{N},[.,.])$ be a   n-ary-Nambu-Lie superalgebra. A one -parameter formal Lie super deformation of $\mathcal{N}$ is given by the $\mathbb{K}[[t]]$-multilinear map $[.,\dots,.]_{t}:\mathcal{N}^{n-1}[[t]]\times\mathcal{N}[[t]]\longrightarrow\mathcal{N}[[t]]$ of the form
$$\displaystyle [.,\dots,.]_{t}=\sum_{i\geq 0} t^i[.,\dots,.]_{i} ,$$
where each $[.,\dots,.]_{i}$ is a even multiilinear  map $[.,\dots,.]_{i}:\mathcal{N}^{n-1}\times\mathcal{N}\longrightarrow\mathcal{N}$ (extended to be $\mathbb{K}[[t]]$-multilinear),
$[.,\dots,.]=[.,\dots,.]_{0}$ and satisfying the following conditions
\begin{eqnarray}
&&[x_{1},\dots,x_{i},\dots,x_{j},\dots,x_{n}]_{t}=(-1)^{j-i+|x_{i}|(|x_{i+1}|+\dots+|x_{j-1}|)}[x_{1},\dots,x_{j},\dots,x_{i},\dots,x_{n}]_{t},\\
&&[x,[y,z]_{t}]_{t}- [[x,y]_{L},z]_{t}-(-1)^{|y||x|} [y,[x,z]_{t}]_{t}=0\label{djacobie}
\end{eqnarray}
for all homogeneous elements $x,\ y\in \mathcal{N}^{n-1}$ and $z\in \mathcal{N}$.
The super deformation is said to be of order $k$ if $\displaystyle [.,\dots,.]_{t}=\sum_{i= 0}^{k} t^i[.,\dots,.]_{i}$.\\
Given two deformation $\mathcal{N}_{t}=(\mathcal{N},[.,\dots,.]_{t})$ and   $\mathcal{N}_{t}'=(\mathcal{N}',[.,\dots,.]_{t}')$ of $\mathcal{N}$ where $\displaystyle [.,\dots,.]_{t}=\sum_{i\geq 0} t^i[.,\dots,.]_{i}$ and  $\displaystyle [.,\dots,.]_{t}'=\sum_{i\geq 0} t^i[.,\dots,.]_{i}'$ with  $\displaystyle [.,\dots,.]_{0}=[.,.]_{0}'=[.,\dots,.]$. We say that $\mathcal{N}_{t}$ and $\mathcal{N}_{t}' $ are equivalent  if there exists a formal automorphism $\displaystyle \Phi_{t}=\sum_{i= 0}^{k} \Phi_{i}t^{i}$ where $\Phi_{i}\in End(\mathcal{N})$ and
$\Phi_{0}=id_{\mathcal{N}}$, such that $$\displaystyle \Phi_{t}([x,z]_{t})=[\Phi_{t}(x),\Phi_{t}(z)]'_{t}.$$
A deformation  $\mathcal{N}_{t}$ is said to be trivial if and only if  $\mathcal{N}_{t}$ is equivalent to $\mathcal{N}$ ( viewd as superalgebra on $\mathcal{N}[[t]].)$\\
The identity (\ref{djacobie}) is called a deformation equation and it is equivalent to
$$\displaystyle \sum_{i\geq0,j\geq0}\bigg([x,[y,z]_{i}]_{j}- [[x,y]_{L},z]_{t}-(-1)^{|y||x|} [y,[x,z]_{i}]_{j}\bigg)t^{i+j}=0;$$
i.e
$$\displaystyle \sum_{i\geq0,s\geq0}\bigg([x,[y,z]_{i}]_{s-i}- [[x,y]_{L},z]_{s-i}-(-1)^{|y||x|} [y,[x,z]_{i}]_{s-i}\bigg)t^{s}=0,$$
or
$$\displaystyle \sum_{s\geq0}t^{s}\sum_{i\geq0}\bigg([x,[y,z]_{i}]_{s-i}- [[x,y]_{L},z]_{s-i}-(-1)^{|y||x|} [y,[x,z]_{i}]_{s-i}\bigg)=0.$$
The deformation equation is equivalent
 to the follwing infinite system
\begin{equation}
\displaystyle \sum_{i=0}^{s}\bigg([x,[y,z]_{i}]_{s-i}- [[x,y]_{L},z]_{s-i}-(-1)^{|y||x|} [y,[x,z]_{i}]_{s-i}\bigg)=0. \label{defordelta}
\end{equation}
In particular, for $s=0$ we have $[x,[y,z]_{0}]_{0}- [[x,y]_{L},z]_{0}-(-1)^{|y||x|} [y,[x,z]_{0}]_{0}$ wich is the super Jacobie identity of $\mathcal{N}$.\\
The equation for $s=1$, is equivalent to $\delta^{2}_{0}([.,\dots,.]_1)=0.$ Then $[.,\dots,.]_1$ is a $2$-cocycle.\\
For $s\geq2$, the identities (\ref{defordelta}) are equivalent to:
\begin{equation*}
    \delta^{2}([.,\dots,.]_{s}) (x,y,z)=-\sum_{i=1}^{s-1}\bigg([x,[y,z]_{i}]_{s-i}- [[x,y]_{L},z]_{s-i}-(-1)^{|y||x|}, [y,[x,z]_{i}]_{s-i}\bigg).
\end{equation*}
 \end{defn}
One may also prove
\begin{thm}
Let $( \mathcal{N},[.,\dots,.])$ be a   n-ary-Nambu-Lie superalgebra and  $\mathcal{N}_{t}=( \mathcal{N},[.,\dots,.]_{t})$ be a one-parameter formal deformation of $ \mathcal{N}$, where  $\displaystyle [.,\dots,.]_{t}=\sum_{i\geq 0} t^i[.,\dots,.]_{i}$. Then there exists an equivalent deformation  $\mathcal{N}_{t}'=( \mathcal{N},[.,\dots,.]_{t}')$ where $\displaystyle [.,\dots,.]_{t}'=\sum_{i\geq 0} t^i[.,\dots,.]_{i}' $ such that $[.,\dots,.]_{i}'\in Z^{2}(\mathcal{N},\mathcal{N})$ and doesn't belong to $B^{2}(\mathcal{N},\mathcal{N})$.\\

Hence, if $H^{2}(\mathcal{N},\mathcal{N})=0$ then every formal deformation is equivalent to a trivial deformation. The n-ary-Nambu-Lie superalgebra is called rigid.
\end{thm}

\section{ Cohomology of the super $w_{\infty}$ $3$-algebra}
The generators of  CHOVW algebra are given by
 \begin{eqnarray*}
L_{m}^{i}&=&(-1)^i\lambda^{i-\frac{1}{2}} z^{n+i}\frac{\partial^{i}}{\partial z^{i}},\\
\overline{L}_{m}^{i}&=&(-1)^i\lambda^{i+\frac{3}{2}} z^{n+i}\theta \frac{\partial}{\partial \theta}\frac{\partial^{i}}{\partial z^{i}},\\
h_{r}^{\alpha+\frac{1}{2}}&=&(-1)^{\alpha+1}\lambda^{\alpha+\frac{1}{2}} z^{r+\alpha}\frac{\partial}{\partial \theta}\frac{\partial^{\alpha}}{\partial z^{\alpha}},\\
\overline{h}_{r}^{\alpha+\frac{1}{2}}&=&(-1)^{\alpha+1}\lambda^{\alpha+\frac{1}{2}} z^{r+\alpha}\theta \frac{\partial^{\alpha}}{\partial z^{\alpha}}.
\end{eqnarray*}
The commutation relation is defined by
\begin{eqnarray}\label{marsV}
  [a,b] &=& ab-(-1)^{|a||b|} ba.
\end{eqnarray}
Let us define a super $3$-bracket as follows:
\begin{eqnarray}\label{marsV1}
  [a,b,c] &=&[a,b]c+(-1)^{|a|(|b|+|c|)}[b,c]a+(-1)^{|c|(|b|+|c|)}.
\end{eqnarray}
Using (\ref{marsV}), (\ref{marsV1}) and taking the scaling limit $\lambda \rightarrow 0,$ then we obtain the following super $w_{\infty}$-algebra.

 \begin{eqnarray*}
 &&[L_{m}^{i},L_{n}^{j},L_{k}^{h}]=\big(h(n-m)+j(m-k)+i(k-n)\big)L_{m+n+k}^{i+j+h-1},\\
&&[L_{m}^{i},L_{n}^{j},\overline{L}_{k}^{h}]=\big(h(n-m)+j(m-k)+i(k-n)\big)\overline{L}_{m+n+k}^{i+j+h-1},\\
&&[L_{m}^{i},L_{n}^{j},h_{p}^{\alpha+\frac{1}{2}}]=\big(\alpha(n-m)+j(m-p)+i(p-n)\big)h_{m+n+p}^{i+j+\alpha-1+\frac{1}{2}},\\
&&[L_{m}^{i},L_{n}^{j},\overline{h}_{r}^{\alpha+\frac{1}{2}}]=\big(\alpha(n-m)+j(m-r)+i(r-n)\big)\overline{h}_{m+n+r}^{i+j+\alpha-1+\frac{1}{2}},\\
&&[L_{m}^{i},h_{p}^{\alpha+\frac{1}{2}},\overline{h}_{r}^{\beta+\frac{1}{2}}]=\big(i(p-r)+\alpha(r-m)+\beta(m-p)\big)\overline{L}_{m+r+p}^{i+\alpha+\beta-1}.
\end{eqnarray*}
The other brakets are obtained by supersymmetry or equals 0.\\

This algebra is $\mathbb{Z}_{2}$ graded with
$$\displaystyle w_\infty=(w_\infty)_{0} \oplus (w_\infty)_{1} \ \textrm{where}\  (w_\infty)_{0}=\bigoplus _{n\in \mathbb{Z},i\in \mathbb{N}}<L_{n}^{i},\overline{L}_{n}^{i}>\ \textrm{and}\
              (w_\infty)_{1}=\bigoplus _{n\in \mathbb{Z},i\in \mathbb{N}}<h_{n}^{i+\frac{1}{2}},\overline{h}_{n}^{i+\frac{1}{2}}>.$$
In the following, we describe a super $w_\infty$ 3-algebra obtained in \cite{CHOVW} and we compute its derivations and second cohomology group.\\
\subsection{Derivations of the super $w_{\infty}$ $3$-algebra.}
An even  derivation $D$ (resp. odd) is said of degree $(s,t)$ if there exists $(s,t)\in \mathbb{Z}\times \mathbb{N}$ such that, for all $(m,i)\in \mathbb{Z}\times \mathbb{N}$, we have $D\big(<L_{m}^{i}>\oplus <\overline{L}_{m}^{i}>\big)\subset \big(<L_{m+s}^{i+t}>\oplus <\overline{L}_{m+s}^{i+t}>\big)$ and
 $D\big(<h_{m}^{i+\frac{1}{2}}>\oplus <\overline{h}_{m}^{i+\frac{1}{2}}>\big)\subset \big(<h_{m+s}^{i+t+\frac{1}{2}}>\oplus <\overline{h}_{m+s}^{i+t+\frac{1}{2}}>\big)$
  (resp. $D\big(<L_{m}^{i}>\oplus <\overline{L}_{m}^{i}>\big)\subset \big(<h_{m+s}^{i+t+\frac{1}{2}}>\oplus <\overline{h}_{m+s}^{i+t+\frac{1}{2}}>\big)$ and
 $D\big(<h_{m}^{i+\frac{1}{2}}>\oplus <\overline{h}_{m}^{i+\frac{1}{2}}>\big)\subset \big(<L_{m+s}^{i+t}>\oplus <\overline{L}_{m+s}^{i+t}>\big)$).\\
It easy to check that $\displaystyle Der (w_{\infty})=\bigoplus _{(s,t)\in \mathbb{Z}\times \mathbb{N}}\big( Der (w_{\infty})_{0}^{(s,t)} \oplus Der (w_{\infty})_{1}^{(s,t)}\big)$.\\
Let $f$ be a homogeneous derivation
\begin{eqnarray*}
f([x_{1},x_{2},x_{3}])&=&[f(x_{1}),x_{2},x_{3}]+(-1)^{|f||x_{1}|}[x_{1},f(x_{2}),x_{3}]+
(-1)^{|f|(|x_{1}|+|x_{2}|)}[x_{1},f(x_{2}),f(x_{3})].
\end{eqnarray*}
We deduce that
\begin{eqnarray}
\Big(h(n-m)+j(m-k)+i(k-n)\Big)f(L_{m+n+k}^{i+j+h-1})&=&\Big[f(L_{m}^{i}),L_{n}^{j},L_{k}^{h}\Big]
+\Big[L_{m}^{i},f(L_{n}^{j}),L_{k}^{h}\Big]+\Big[L_{m}^{i},L_{n}^{j},f(L_{k}^{h})\Big],\nonumber \\ \label{der1}
 \end{eqnarray}
 \begin{eqnarray}
\Big(h(n-m)+j(m-k)+i(k-n)\Big)f(\overline{L}_{m+n+k}^{i+j+h-1})&=&
\Big[f(L_{m}^{i}),L_{n}^{j},\overline{L}_{k}^{h}\Big]+\Big[L_{m}^{i},f(L_{n}^{j}),\overline{L}_{k}^{h}\Big]
 +\Big[L_{m}^{i},L_{n}^{j},f(\overline{L}_{k}^{h})\Big]\nonumber\\ \label{der2}
 \end{eqnarray}
 \begin{eqnarray}
\Big(\alpha(n-m)+j(m-r)+i(r-n)\Big) f(h_{m+n+r}^{i+j+\alpha-1+\frac{1}{2}})
 &=&\Big[f(L_{m}^{i}),L_{n}^{j},h_{r}^{\alpha+\frac{1}{2}}\Big]
 +\Big[L_{m}^{i},f(L_{n}^{j}),h_{r}^{\alpha+\frac{1}{2}}\Big]\nonumber\\
 &&+\Big[L_{m}^{i},L_{n}^{j},f(h_{r}^{\alpha+\frac{1}{2}})\Big] \label{der3}
 \end{eqnarray}
 \begin{eqnarray}
 \Big(\alpha(n-m)+j(m-r)+i(r-n)\Big)f(\overline{h}_{m+n+r}^{i+j+\alpha-1+\frac{1}{2}})
 &=&\Big[f(L_{m}^{i}),L_{n}^{j},h_{r}^{\alpha+\frac{1}{2}}\Big]
 +\Big[L_{m}^{i},f(L_{n}^{j}),h_{r}^{\alpha+\frac{1}{2}}\Big]\nonumber \\
&& +\Big[L_{m}^{i},L_{n}^{j},f(\overline{h}_{r}^{\alpha+\frac{1}{2}})\Big] \label{der4}
 \end{eqnarray}
  \begin{eqnarray}
  \Big(\alpha(m-n)+j(r-m)+i(n-r) \Big)f(\overline{L}_{m+n+r}^{i+j+\alpha-1})
 &=& \Big[f(L_{m}^{i}),h_{n}^{j+\frac{1}{2}},\overline{h}_{r}^{\alpha+\frac{1}{2}} \Big]
 + \Big[L_{m}^{i},f(h_{n}^{j+\frac{1}{2}}),\overline{h}_{r}^{\alpha+\frac{1}{2}} \Big]\nonumber \\
 &&+(-1)^{|f|} \Big[L_{m}^{i},h_{n}^{j+\frac{1}{2}},f(\overline{h}_{r}^{\alpha+\frac{1}{2}}) \Big] \label{der5}
 \end{eqnarray}
\subsubsection{ Even derivations of the super $w_{\infty}$ $3$-algebra}
Let $f$ be an even derivation of degree $(s,t)$,
$$f(L_{m,s}^{i,t})=a_{m,s}^{i,t}L_{m+s}^{i+t}+b_{m,s}^{i,t}\overline{L}_{m+s}^{i+t},$$
\begin{eqnarray*}
f(\overline{L}_{m}^{i})&=&c_{m,s}^{i,t}L_{m+s}^{i+t}+d_{m,s}^{i,t}\overline{L}_{m+s}^{i+t},
   \end{eqnarray*}
$$f(h_{r}^{\alpha+\frac{1}{2}})
=e_{r,s}^{\alpha,t}h_{r+s}^{\alpha+t+\frac{1}{2}}+f_{r,s}^{\alpha,t}\overline{h}_{r}^{\alpha+\frac{1}{2}}$$ and
$$f(\overline{h}_{r}^{\beta+\frac{1}{2}})
=x_{r,s}^{\beta,t}h_{r+s}^{\beta+t+\frac{1}{2}}+y_{r,s}^{\beta,t}\overline{h}_{r+s}^{\beta+t+\frac{1}{2}}.$$
By (\ref{der1}) we have

   \begin{eqnarray}
 &&a_{m+n+k,s}^{i+j+h-1,t}\Big(h(n-m)+j(m-k)+i(k-n)\Big)\nonumber\\
 &=&a_{m,s}^{i,t}\Big(h(n-m)+j(m-k)+i(k-n)-hs+js+t(k-n)\Big)\nonumber\\\
 &&+a_{n,s}^{j,t}\Big(h(n-m)+(j)(m-k)+i(k-n)+hs+t(m-k)-is   \Big)\nonumber\ \\
 &&+a_{k,s}^{h,t}\Big(h(n-m)+j(m-k)+i(k-n) +t(n-m)-js+is          \Big),\label{equ13j}
 \end{eqnarray}
 and
  \begin{eqnarray}
 &&b_{m+n+k,s}^{i+j+h,t}\Big(h(n-m)+j(m-k)+i(k-n)\Big)\nonumber\\
 &=&b_{m,s}^{i,t}\Big(h(n-m)+j(m-k)+i(k-n)+t(k-n)-hs+js\Big)\nonumber\\
   &&+b_{n,s}^{j,t}\Big(h(n-m)+j(m-k)+i(k-n)+t(m-k)+hs-is\Big)\nonumber\\
   &&+b_{k,s}^{h,t}\Big(h(n-m)+j(m-k)+i(k-n)+t(n-m)-js+is\Big).\label{eequ13j}
 \end{eqnarray}
By (\ref{der2}), we have
\begin{eqnarray}
 \Big(h(n-m)+j(m-k)+i(k-n)\Big)c_{m+n+k,s}^{i+j+h-1,t}&=&\Big(h(n-m)+j(m-k)+i(k-n)+ t(n-m)-js+is\Big)c_{k,s}^{h,t}\nonumber\\\label{eeequ13j}
 \end{eqnarray}
 and
 \begin{eqnarray}
&&\Big(h(n-m)+j(m-k)+i(k-n)\Big)d_{m+n+k,s}^{i+j+h-1,t}\nonumber\\
&=&\Big( h(n-m)+j(m-k)+i(k-n)      -hs+js+t(k-n)  \Big)a_{m,s}^{i,t}\nonumber\\
&&+\Big( h(n-m)+j(m-k)+i(k-n)  +         hs+t(m-k)-is  \Big)a_{n,s}^{j,t}\nonumber\\
&&+\Big( h(n-m)+j(m-k)+i(k-n)  +       t(n-m)-js+is   \Big)d_{k,s}^{h,t}.\label{deq13ja}
 \end{eqnarray}
 By (\ref{der3}) and (\ref{der4}), we obtain, exactly, the same equation as  (\ref{deq13ja}).\\
By (\ref{der5}), we obtain
 \begin{eqnarray}
&&\Big(\alpha(m-n)+j(r-m)+i(n-r)\Big)d_{m+n+r,s}^{i+j+h-1,t}\nonumber\\
&=&\Big(\alpha(m-n)+j(r-m)+i(n-r)           -j s+\alpha s+t(n-r)  \Big)a_{m,s}^{i,t}\nonumber\\
&&+\Big( \alpha(m-n)+j(r-m)+i(n-r) -       \alpha s+t(r-m)+is          \Big)e_{n,s}^{j,t}\nonumber\\
&&+\Big( \alpha(m-n)+j(r-m)+i(n-r)   +       t(m-n)+js-is  \Big)y_{r,s}^{\alpha,t},\label{eq13ja}
 \end{eqnarray}
and
 \begin{eqnarray}
\Big(\alpha(m-n)+j(r-m)+i(n-r)\Big)c_{m+n+r,s}^{i+j+\alpha-1,t}
&=&0.\label{eeeq13ja}
 \end{eqnarray}
Taking $m=n=i=0,\ j=1$ (resp. $m=1,\ n=-1,\ r=0,\ i=1,\ j=0$), we obtain $c_{r,s}^{\alpha,t}=0, \ \forall (r,s)\in \mathbb{Z}\times \mathbb{N}.$
\begin{prop}
\begin{itemize}
\item If $s+ 2t\neq 0$ we have
 \begin{eqnarray*}
 Der (w_{\infty})_{0}^{(s,t)}&=&
 <ad(L_{1+s}^{t},L_{-1}^{1})>
 \oplus <ad(L_{1}^{0},L_{-1+s}^{1+t})>
 \oplus <ad(\overline{L}_{1+s}^{t},L_{-1}^{1})>
 \oplus <ad(L_{1}^{0},\overline{L}_{-1+s}^{1+t})>,
\end{eqnarray*}
\item If $s+ 2t= 0$ and $t\neq 0$ we have
 \begin{eqnarray*}
 Der (w_{\infty})_{0}^{(s,t)}&=&
 <ad(L_{1+s}^{1+t},L_{-1}^{0})>
 \oplus <ad(L_{1}^{1},L_{-1+s}^{t})>
 \oplus <ad(\overline{L}_{1+s}^{1+t},L_{-1}^{0})>
 \oplus <ad(L_{1}^{1},\overline{L}_{-1+s}^{t})>.
\end{eqnarray*}
\item  If $s+ 2t= 0$ and $t= 0$ we have
 \begin{eqnarray*}
 Der (w_{\infty})_{0}^{(0,0)}
 &=&<ad(L_{-1}^{1},L_{1}^{0})>
 \oplus <ad(L_{0}^{1},L_{0}^{0})>
 \oplus <ad(h_{0}^{1+\frac{1}{2}},\overline{h}_{0}^{\frac{1}{2}})>
 \oplus <ad(h_{1}^{0+\frac{1}{2}},\overline{h}_{-1}^{1+\frac{1}{2}})>\\
 &&\oplus <\varphi_{1}>
 \oplus < \varphi_{2}>,
\end{eqnarray*}
where
  \begin{eqnarray*}
  \varphi_{1}(L_{k}^{h})&=& \varphi_{1}(h_{k}^{h+\frac{1}{2}})=0,\  \varphi_{1}(\overline{L}_{k}^{h})=\overline{L}_{k}^{h}, \ \varphi_{1}(\overline{h}_{k}^{\alpha+\frac{1}{2}})=\overline{h}_{k}^{\alpha+\frac{1}{2}}.\\
    \varphi_{2}(L_{k}^{h})&=& \varphi_{2}(\overline{h}_{k}^{\alpha+\frac{1}{2}})=0,\  \varphi_{2}(\overline{L}_{k}^{h})=\overline{L}_{k}^{h}, \ \varphi_{2}(h_{k}^{\alpha+\frac{1}{2}})=h_{k}^{\alpha+\frac{1}{2}}.\\
   \end{eqnarray*}
\end{itemize}
 \end{prop}
 \begin{proof}

Taking $m=1,\ n=-1,\ i=0,\ j=1$ in (\ref{der1}) we obtain
 \begin{eqnarray*}
(-2h+1-k)f(L_{k}^{h})&=&
a_{1,s}^{0,t}ad(L_{1+ s}^{t},L_{-1}^{1})(L_{k}^{h})+b_{1,s}^{0,t}ad(\overline{L}_{1+ s}^{t},L_{-1}^{1})(L_{k}^{h})\\
&&
+a_{-1,s}^{1,t}ad(L_{1}^{0},L_{-1+s}^{1+t})(L_{k}^{h})+b_{-1,s}^{1,t}ad(L_{1}^{0},\overline{L}_{-1+s}^{1+t})(L_{k}^{h})\\
&&+(-2h-2t+1-k-s)f(L_{k}^{h}),
 \end{eqnarray*}
which gives ( if $s+2t\neq0$)
 \begin{eqnarray*}
f(X_{k}^{h})&=&\frac{a_{1,s}^{0,t}}{s+2t}ad(L_{1+ s}^{t},L_{-1}^{1})(X_{k}^{h})+\frac{b_{1,s}^{0,t}}{s+2t}ad(\overline{X}_{1+ s}^{t},L_{-1}^{1})(X_{k}^{h})\\
&&
+\frac{a_{-1,s}^{1,t}}{s+2t}ad(L_{1}^{0},L_{-1+s}^{1+t})(X_{k}^{h})+\frac{b_{-1,s}^{1,t}}{s+2t}ad(L_{1}^{0},\overline{L}_{-1+s}^{1+t})(X_{k}^{h})\ \textrm{for any} \ X_{k}^{h}\in\ w_{\infty}\\
 \end{eqnarray*}
Taking $m=1,\ n=-1,\ i=1,\ j=0$ in (\ref{der1}) we obtain
 \begin{eqnarray*}
(-2h+1+k)f(X_{k}^{h})
&=&a_{1,s}^{1,t}ad(L_{1+ s}^{1+t},L_{-1}^{0})(X_{k}^{h})+b_{1,s}^{1,t}ad(\overline{L}_{1+ s}^{1+t},L_{-1}^{0})(X_{k}^{h})\\
&&
+a_{-1,s}^{0,t}ad(L_{1}^{1},L_{-1+s}^{t})(X_{k}^{h})+b_{-1,s}^{0,t}ad(L_{1}^{1},\overline{L}_{-1+s}^{t})(X_{k}^{h})\\
&&+(-2h-2t+1+k+s)f(X_{k}^{h}),
 \end{eqnarray*}
which gives
 \begin{eqnarray*}
f(X_{k}^{h})&=&\frac{a_{1,s}^{1,t}}{4t}ad(L_{1+ s}^{1+t},L_{-1}^{0})(X_{k}^{h})+\frac{b_{1,s}^{0,t}}{4t}ad(\overline{L}_{1+ s}^{1+t},L_{-1}^{0})(X_{k}^{h})\\
&&
+\frac{a_{-1,s}^{0,t}}{4t}ad(L_{1}^{1},L_{-1+s}^{t})(X_{k}^{h})+\frac{b_{-1,s}^{1,t}}{4t}ad(L_{1}^{1},\overline{L}_{-1+s}^{t})(X_{k}^{h}) \     \textrm{if}  \ s+2t=0 \    \textrm{and}\  t\neq0. \\
 \end{eqnarray*}
 Taking $s=0,\ t=0,$ $m=0,\ n=0,\ i=2,\ j=0$  in (\ref{equ13j}), we obtain
   \begin{eqnarray}
 a_{k,0}^{h,0}
 &=&ha_{0,0}^{2,0}
 +ha_{0,0}^{0,0}
 +a_{k,0}^{0,0}, \forall k\in \mathbb{Z}^{*}.\label{Jderiv}
 \end{eqnarray}
 Taking $s=0,\ t=0,$ $h=0,\ m=1,\ i=0,\ n=0,\ j=1$ in (\ref{equ13j}), we obtain
   \begin{eqnarray*}
 (1-k)a_{k+1,0}^{0,0}
 &=& (1-k)a_{1,0}^{0,0}
 + (1-k)a_{0,0}^{1,0}
 + (1-k)a_{k,0}^{0,0}.
 \end{eqnarray*}
 We deduce that

 $$a_{k,0}^{0,0}
 = (k-1)a_{1,0}^{0,0}
 + (k-1)a_{0,0}^{1,0}+a_{1,0}^{0,0} \ \forall k\in\mathbb{Z}^{*}.$$
Then
$a_{k,0}^{h,0}
 =ha_{0,0}^{2,0}
 +ha_{0,0}^{0,0}
 +(k-1)a_{1,0}^{0,0}
 + (k-1)a_{0,0}^{1,0}+a_{1,0}^{0,0}$. So,
     \begin{eqnarray*}
a_{1,0}^{2,0}
 &=&2a_{0,0}^{2,0}
 +2a_{0,0}^{0,0}
 +a_{1,0}^{0,0}=2a_{0,0}^{2,0}
 -2a_{0,0}^{1,0}
 +a_{1,0}^{0,0}
 \end{eqnarray*}
Taking  $m=1,\ n=0,\,\ k=0,\ i=0,\ j=1,\ h=2$ in (\ref{equ13j}), we obtain
      \begin{eqnarray*}
a_{1,0}^{2,0}
 &=&a_{0,0}^{1,0}
 +a_{1,0}^{0,0}
 +a_{0,0}^{2,0}
 \end{eqnarray*}
 So $$a_{0,0}^{2,0}=3a_{0,0}^{1,0}.$$
We deduce that
      \begin{eqnarray}
a_{k,0}^{h,0}
 &=&(2h+k-1)a_{0,0}^{1,0}+ka_{1,0}^{0,0}.\label{a}
 \end{eqnarray}
Then
      \begin{eqnarray*}
b_{k,0}^{h,0}
 &=&(2h+k-1)b_{0,0}^{1,0}+kb_{1,0}^{0,0}.
 \end{eqnarray*}

Taking $s=0,\ t=0,$  $m=1,\ n=-1,\ i=j=1$ in (\ref{eq13ja}), we obtain
   \begin{eqnarray*}
d_{k,0}^{h,0}&=&(h-2)(a_{1,0}^{1,0}
+a_{-1,0}^{1,0})+d_{k,0}^{2,0}
 \end{eqnarray*}
 Since $a_{1,0}^{1,0}
+a_{-1,0}^{1,0}=a_{1,0}^{0,0}$ we deduce that
$d_{k,0}^{h,0}=2(h-2)a_{0,0}^{1,0}
+d_{k,0}^{2,0}.$\\
Taking $s=0,\ t=0,$  $m=1,\ n=-0,\ i=0,\ j=1,\ h=2$ in (\ref{eq13ja}) we obtain

  \begin{eqnarray*}
-(k+1)d_{k+1,0}^{2,0}&=&-(k+1)a_{1,0}^{0,0}
-(k+1)a_{0,0}^{1,0}
-(k+1)d_{k,0}^{3,0}.
 \end{eqnarray*}
One can deduce that
     \begin{eqnarray}
d_{k,0}^{h,0}
&=&(2h+k-1)a_{0,0}^{1,0}+ka_{1,0}^{0,0}+d_{0,0}^{2,0}-3a_{0,0}^{1,0}, \ \forall k\in \mathbb{Z}\setminus\{-1\}.\label{d}
 \end{eqnarray}

Then,
      \begin{eqnarray}
e_{k,0}^{h,0}
&=&(2h+k-1)a_{0,0}^{1,0}+ka_{1,0}^{0,0}+e_{0,0}^{2,0}-3a_{0,0}^{1,0},\ \forall k\in \mathbb{Z}\setminus\{-1\},\label{e}
 \end{eqnarray}
 and
    \begin{eqnarray}
 \  y_{r,0}^{\alpha,0}=(2\alpha+r-1)a_{0,0}^{1,0}+ra_{1,0}^{0,0}+y_{0,0}^{2,0}-3a_{0,0}^{1,0},\     \forall r\in \mathbb{Z}\setminus\{-1\}.\label{y}
  \end{eqnarray}
Taking $m=1,\ n=-1,\ i=j=1$ in  (\ref{der3}), we obtain $x_{r,0}^{\alpha,0}=x_{k}^{2}=0,\ \forall h\in \mathbb{N}^{*}\setminus\{1\}.$\\
So,
$ f_{r,0}^{\alpha,0}=f_{r}^{2}=0,       \forall \alpha\in \mathbb{N}^{*}\setminus\{1\}. $\\
Using (\ref{a}),\ (\ref{d}),\ (\ref{e}) and (\ref{y}) in (\ref{eq13ja}), we obtain $d_{0,0}^{2,0}=-3a_{0,0}^{1,0}+e_{0,0}^{2,0}+y_{0,0}^{2,0}$.\\

Therefore,
  \begin{eqnarray*}
f&=&a_{0,0}^{1,0}ad(L_{-1}^{1},L_{1}^{0})+a_{1,0}^{0,0}ad(L_{0}^{1},L_{0}^{0})+b_{0,0}^{1,0}ad(h_{1}^{0},\overline{h}_{-1}^{1})
-b_{1,0}^{0,0}ad(h_{0}^{1},\overline{h}_{0}^{0})+(d_{0,0}^{2,0}-3a_{0,0}^{1,0})\varphi_{1}+(e_{0,0}^{2,0}-3a_{0,0}^{1,0})\varphi_{2},
 \end{eqnarray*}
where,
  \begin{eqnarray*}
  \varphi_{1}(L_{k}^{h})&=& \varphi_{1}(h_{k}^{h+\frac{1}{2}})=0,\  \varphi_{1}(\overline{L}_{k}^{h})=\overline{L}_{k}^{h}, \ \varphi_{1}(\overline{h}_{k}^{\alpha+\frac{1}{2}})=\overline{h}_{k}^{\alpha+\frac{1}{2}},\\
    \varphi_{2}(L_{k}^{h})&=& \varphi_{2}(\overline{h}_{k}^{\alpha+\frac{1}{2}})=0,\  \varphi_{2}(\overline{L}_{k}^{h})=\overline{L}_{k}^{h}, \ \varphi_{2}(h_{k}^{\alpha+\frac{1}{2}})=h_{k}^{\alpha+\frac{1}{2}}.\\
   \end{eqnarray*}
\end{proof}
\subsection{Odd derivation}
Let $f$ be an odd derivation of degree $(s,t)$,
 \begin{eqnarray}
f(L_{m,s}^{i,t})&=&a_{m,s}^{i,t} h_{m+s}^{i+t\frac{1}{2}}+b_{m,s}^{i,t}\overline{h}_{m+s}^{i+t+\frac{1}{2}}    ,\nonumber\\
f(\overline{L}_{m}^{i})&=&c_{m,s}^{i,t}h_{m+s}^{i+t}+d_{m,s}^{i,t}\overline{h}_{m+s}^{i+t},\label{cfevrier}\\\
f(h_{r}^{\alpha+\frac{1}{2}})
&=&e_{r,s}^{\alpha,t}l_{r+s}^{\alpha+t}+f_{r,s}^{\alpha,t}\overline{L}_{r+s}^{\alpha+t}
\nonumber\\
f(\overline{h}_{r}^{\beta+\frac{1}{2}})&=&x_{r,s}^{\beta,t}l_{r+s}^{\beta+t}+y_{r,s}^{\beta,t}\overline{L}_{r+s}^{\beta+t}\nonumber\\.\nonumber
 \end{eqnarray}

\begin{prop}
\begin{itemize}
\item If $s+2t\neq 0$, then we have
 \begin{eqnarray*}
&& Der (w_{\infty})_{1}^{(s,t)}\\
&=&
 <ad(h_{1+ s}^{t+\frac{1}{2}},L_{-1}^{1})>
 \oplus <ad(\overline{h}_{1+ s}^{t+\frac{1}{2}},L_{-1}^{1})>
  \oplus <ad(L_{1}^{0},h_{-1+s}^{1+t+\frac{1}{2}})>\oplus <ad(L_{1}^{0},\overline{h}_{-1+s}^{1+t+\frac{1}{2}})>.
\end{eqnarray*}
\item If $s+2t=0$  and $t\neq0$ we have
 \begin{eqnarray*}
  &&Der (w_{\infty})_{1}^{(s,t)}\\
  &=&
<ad(  h_{1+ s}^{1+t+\frac{1}{2}},L_{-1}^{0}     )>
\oplus <ad(\overline{h}_{1+ s}^{1+t+\frac{1}{2}},L_{-1}^{0})>
\oplus
 <ad(L_{1}^{1},h_{-1+s}^{t+\frac{1}{2}})>\oplus <ad(L_{1}^{1},\overline{h}_{-1+s}^{t+\frac{1}{2}})>.\\
 \end{eqnarray*}
\item If $s+2t=0$  and $t=0$
   \begin{eqnarray*}
 Der (w_{\infty})_{1}^{(0,0)}&=&<ad(L_{1}^{0},h_{-1}^{1+\frac{1}{2}})>\oplus <ad(L_{1}^{0},\overline{h}_{-1}^{1+\frac{1}{2}})>\oplus <ad(L_{0}^{0},h_{0}^{1+\frac{1}{2}})>\oplus <ad(L_{0}^{0},\overline{h}_{0}^{1+\frac{1}{2}})>
\end{eqnarray*}

\end{itemize}
 \end{prop}
\begin{proof}

Setting $m=1,\ n=-1,\ i=0,\ j=1$ in (\ref{der1}), one has
 \begin{eqnarray*}
(-2h+1-k)f(L_{k}^{h})
&=&a_{1,s}^{0,t}ad(h_{1+ s}^{t+\frac{1}{2}},L_{-1}^{1})(L_{k}^{h})+b_{1,s}^{0,t}ad(\overline{h}_{1+ s}^{t+\frac{1}{2}},L_{-1}^{1})(L_{k}^{h})\\
&&
+a_{-1,s}^{1,t}ad(L_{1}^{0},h_{-1+s}^{1+t+\frac{1}{2}})(L_{k}^{h})+b_{-1,s}^{1,t}ad(L_{1}^{0},\overline{h}_{-1+s}^{1+t+\frac{1}{2}})(L_{k}^{h})\\
&&+(-2h-2t+1-k-s)f(L_{k}^{h}),
 \end{eqnarray*}

which gives
 \begin{eqnarray*}
&&f(L_{k}^{h})=\frac{a_{1,s}^{0,t}}{s+2t}ad(h_{1+ s}^{t+\frac{1}{2}},L_{-1}^{1})(L_{k}^{h})+\frac{b_{1,s}^{0,t}}{s+2t}ad(\overline{h}_{1+ s}^{t+\frac{1}{2}},L_{-1}^{1})(L_{k}^{h})\\
&&
+\frac{a_{-1,s}^{1,t}}{s+2t}ad(L_{1}^{0},h_{-1+s}^{1+t+\frac{1}{2}})(L_{k}^{h})
+\frac{b_{-1,s}^{1,t}}{s+2t}ad(L_{1}^{0},\overline{h}_{-1+s}^{1+t+\frac{1}{2}})(L_{k}^{h})\ \
\textrm{if}\ s+2t\neq0.
 \end{eqnarray*}
Furthermore , taking $m=1,\ n=-1,\ i=1,\ j=0$ in (\ref{der1}), one has
 \begin{eqnarray*}
f(L_{k}^{h})&=&\frac{a_{1,s}^{1,t}}{4t}ad(  h_{1+ s}^{1+t+\frac{1}{2}},L_{-1}^{0}     )(L_{k}^{h})
+\frac{b_{1,s}^{0,t}}{4t}ad(\overline{h}_{1+ s}^{1+t+\frac{1}{2}},L_{-1}^{0})(L_{k}^{h})\\
&&
+\frac{a_{-1,s}^{0,t}}{4t}ad(L_{1}^{1},h_{-1+s}^{t+\frac{1}{2}})(L_{k}^{h})+
\frac{b_{-1,s}^{1,t}}{4t}ad(L_{1}^{1},\overline{h}_{-1+s}^{t+\frac{1}{2}})(L_{k}^{h})\ \textrm{if}\ \ t\neq0.\\
 \end{eqnarray*}

 By (\ref{der1}) and (\ref{cfevrier}) we obtain, exactly, the same equation as    (\ref{equ13j}) and (\ref{eequ13j}).\\
 By (\ref{der2})  and (\ref{cfevrier})  we obtain, exactly, the same equation as  ( \ref{eeequ13j}) and (\ref{deq13ja}).\\
  By (\ref{der3}) and (\ref{cfevrier}) we obtain (\ref{deq13ja}) and (\ref{eeequ13j}).\\
  By (\ref{der4}) and (\ref{cfevrier}) we obtain (\ref{deq13ja}) and (\ref{eeequ13j}).\\
Therefore,
      \begin{eqnarray*}
a_{k,0}^{h,0}
 &=&(2h+k-1)a_{0,0}^{1,0}+ka_{1,0}^{0,0}
 \end{eqnarray*}
       \begin{eqnarray*}
b_{k,0}^{h,0}
 &=&(2h+k-1)b_{0,0}^{1,0}+kb_{1,0}^{0,0}.
 \end{eqnarray*}
      \begin{eqnarray*}
d_{k,0}^{h,0}
&=&(2h+k-1)a_{0,0}^{1,0}+ka_{1,0}^{0,0}+d_{0,0}^{2,0}-3a_{0,0}^{1,0}, \ \forall k\in \mathbb{Z}\setminus\{-1\}.
 \end{eqnarray*}
We deduce that

   \begin{eqnarray*}
 f&=&a_{0,0}^{1,0}ad(L_{1}^{0},h_{-1}^{1+\frac{1}{2}})+a_{1,0}^{0,0} ad(L_{1}^{0},\overline{h}_{-1}^{1+\frac{1}{2}})+ b_{0,0}^{1,0} ad(L_{0}^{0},h_{0}^{1+\frac{1}{2}})+ b_{1,0}^{0,0}ad(L_{0}^{0},\overline{h}_{0}^{1+\frac{1}{2}})
 \end{eqnarray*}

\end{proof}

\subsection{Cohomology space $H_{0}^{2}(w_{\infty},\mathbb{C})$ of  $w_{\infty}$}
In the  following we describe the  cohomology space $H_{0}^{2}(w_{\infty},\mathbb{C})$ . We denote by $[f]$ the cohomology class of an element $f$.
\begin{thm}
The second cohomology of the super $w_{\infty}$ $3$-algebra with values in the trivial module vanishes, i.e
\begin{eqnarray*}
H_{0}^{2}(w_{\infty},\mathbb{C})&=&\{0\}.
\end{eqnarray*}
\end{thm}
\begin{proof}
For all $f\in Z^{1}_{0}(w_{\infty},\mathbb{C})$ (resp. $f\in Z^{2}_{0}(w_{\infty},C)$)  we have,  respectively,
\begin{eqnarray}\label{1cycle}
\delta^{1}(f)(x_{0},z)&=&f([x_{1},x_{2},x_{3}])
=0,
\end{eqnarray}
\begin{eqnarray}
\delta^{2}(f)(x_{0},x_{1},z)
&=&-
f\Big( [x_{0},x_{1}],z\Big) -(-1)^{|x_{0}||x_{1}|}
f\Big(x_{1},ad(x_{0})(z)\Big)+
f\Big(x_{0},ad(x_{1})(z)\Big)=0.\nonumber\\ \label{ddd}
\end{eqnarray}

\textbf{ Case 1 $\displaystyle(X_{p}^{v},X_{q}^{l},X_{k}^{h})\in \bigg\{(L_{m}^{i},L_{n}^{j},L_{k}^{h}), (L_{m}^{i},L_{n}^{j},\overline{L}_{k}^{h}),(L_{m}^{i},L_{n}^{j},h_{r}^{\alpha+\frac{1}{2}}),(L_{m}^{i},L_{n}^{j},\overline{h}_{r}^{\alpha+\frac{1}{2}})
,(h_{r}^{\alpha+\frac{1}{2}},L_{m}^{i},\overline{h}_{s}^{\beta+\frac{1}{2}})\bigg\}$.}
Using (\ref{ddd}), we have
\begin{eqnarray*}
-
f\Big( [L_{m}^{i}\wedge L_{n}^{j},X_{p}^{v}\wedge X_{q}^{l}],X_{k}^{h}\Big) -
f\Big(X_{p}^{v}\wedge X_{q}^{l},ad(L_{m}^{i}\wedge L_{n}^{j})(X_{k}^{h})\Big)+
f\Big(L_{m}^{i}\wedge L_{n}^{j},ad(X_{p}^{v}\wedge X_{q}^{l})(X_{k}^{h})\Big)&=&0.
\end{eqnarray*}
Since $ [x,y]_{l}=\displaystyle \sum_{i=1}^{n-1}(-1)^{|x|(|y_{1}|+\cdots +|y_{i-1}|)}y_{1}\wedge\cdots\wedge ad(x)(y_{i})\wedge\cdots \wedge y_{n-1},$
and $ad(a\wedge b)(z)=[a,b,z]$ then
\begin{eqnarray*}
&&-
f\Big( [L_{m}^{i}\wedge L_{n}^{j},  X_{p}^{v}]\wedge X_{q}^{l} ,X_{k}^{h}\Big)-f\Big(X_{p}^{v},[L_{m}^{i}\wedge L_{n}^{j},X_{q}^{l}],X_{k}^{h} \Big)\\
&&-f\Big(X_{p}^{v}\wedge X_{q}^{l},[L_{m}^{i}\wedge L_{n}^{j},X_{k}^{h}]\Big)
+f\Big(L_{m}^{i}\wedge L_{n}^{j},[X_{p}^{v}\wedge X_{q}^{l},X_{k}^{h}]\Big)=0.
\end{eqnarray*}
So,
\begin{eqnarray}
&&-\bigg(v(n-m)+j(m-\overline{p})+i(\overline{p}-n)\bigg)
                    f\Big( X_{m+n+p}^{i+j+v-1}\wedge X_{q}^{l} ,X_{k}^{h}\Big)\nonumber\\
                    &&-\bigg(\overline{l}(n-m)+j(m-\overline{q})+i(\overline{q}-n)\bigg)
                    f\Big(X_{\overline{p}}^{v}\wedge X_{m+n+\overline{q}}^{i+j+l-1},X_{k}^{h} \Big)\nonumber\\
&&-\bigg(\overline{h}(n-m)+j(m-k)+i(k-n)\bigg)f\Big(X_{p}^{v}\wedge X_{q}^{l},X_{m+n+k}^{i+j+h-1}\Big)\nonumber\\
&&+\bigg(\overline{h}(\overline{q}-\overline{p})+l (\overline{p}-k)+i(k-\overline{q})          \bigg)
f\Big(L_{m}^{i}\wedge L_{n}^{j},X_{p+q+k}^{v+l+h-1}\Big)=0,\label{thomm1}
\end{eqnarray}
where    $\overline{p}$ is the integer part of $p$.

Setting $m=0,\ n=0,\ i=1,\ j=0$, in (\ref{thomm1}), we obtain
$$\displaystyle f( X_{p}^{v}\wedge X_{q}^{l} ,X_{k}^{h}\Big))=\frac{h(q-p)+l (p-k)+(k-q) }{p+q+k} f(L_{0}^{1}\wedge L_{0}^{0},X_{p+q+k}^{v+t+h-1})\ \  (p+q+k\neq 0).$$\\

Setting $m=1,\ n=-1,\ i=1,\ j=0,\ k=-p-q$  (\ref{thomm1}), we obtain
\begin{eqnarray*}
f( X_{p}^{v}\wedge X_{q}^{l} ,X_{-p-q}^{h})
&=&-\frac{h(q-p)+l (2p-q)+(-p-2q)  }{2l+2v+2h-3}
f(L_{1}^{1}\wedge L_{-1}^{0},X_{0}^{v+l+h-1}).
\end{eqnarray*}

Setting $m=1,\ j=0,\ i=1,\ p=1,\ q=-1,\ v=1,\ l=0,\ k=0$  (\ref{thomm1}), we obtain
\begin{eqnarray*}
&&(-2h-1)
f\Big(L_{1}^{1}\wedge L_{-1}^{0},X_{0}^{h}\Big)=0.
\end{eqnarray*}
Then  $f(L_{1}^{1}\wedge L_{-1}^{0},X_{0}^{h})=0$.\\
We deduce that
\begin{eqnarray*}
f( X_{p}^{v}\wedge X_{q}^{l} ,X_{-p-q}^{h})
&=&0.
\end{eqnarray*}
We define an endomorphism $g$ of $w_{\infty}$ by
$g(X_{k}^{h})=-\frac{1}{k}f(L_{1}^{1}\wedge L_{-1}^{0},X_{k}^{h})$ and $g(X_{0}^{h})=0$.\\
By (\ref{1cycle}) we obtain
\begin{eqnarray*}
\delta^{1}(g)(X_{p}^{v}\wedge X_{q}^{l},X_{k}^{h})&=&-g(ad(X_{p}^{v}\wedge X_{q}^{l})(X_{k}^{h}))
=f(L_{m}^{i}\wedge L_{n}^{j},X_{k}^{h}).
\end{eqnarray*}
\textbf{Case 2  $\displaystyle(X_{p}^{v},X_{q}^{l},X_{k}^{h})\notin \bigg\{(L_{m}^{i},L_{n}^{j},L_{k}^{h}), (L_{m}^{i},L_{n}^{j},\overline{L}_{k}^{h}),(L_{m}^{i},L_{n}^{j},h_{r}^{\alpha+\frac{1}{2}}),(L_{m}^{i},L_{n}^{j},\overline{h}_{r}^{\alpha+\frac{1}{2}})
,(h_{r}^{\alpha+\frac{1}{2}},L_{m}^{i},\overline{h}_{s}^{\beta+\frac{1}{2}})\bigg\}$.}
Using (\ref{ddd}), we have
\begin{eqnarray}
&&-\Big(v(n-m)+j(m-p)+i(p-n)\Big)
                    f\Big( X_{m+n+p}^{i+j+v-1}\wedge X_{q}^{l} ,X_{k}^{h}\Big)\nonumber\\
                    &&-\Big(l(n-m)+j(m-q)+i(q-n)\Big)
                    f\Big(X_{p}^{v}\wedge X_{m+n+q}^{i+j+l-1},X_{k}^{h} \Big)\nonumber\\
&&-\Big(h(n-m)+j(m-k)+i(k-n)\Big)f\Big(X_{p}^{v}\wedge X_{q}^{l},X_{m+n+k}^{i+j+h-1}\Big)\label{tthomm1}\\
&&+
f\Big(L_{m}^{i}\wedge L_{n}^{j},0\Big)=0.\nonumber
\end{eqnarray}
Setting $m=0,\ n=0,\ i=1,\ j=0$, in (\ref{tthomm1}), we obtain
$\displaystyle f( X_{p}^{v}\wedge X_{q}^{l} ,X_{k}^{h})=0  \ \  (p+q+k\neq 0).$\\
Setting $m=1,\ n=-1,\ i=1,\ j=0$, in (\ref{tthomm1}), we obtain
$\displaystyle f( X_{p}^{v}\wedge X_{q}^{l} ,X_{k}^{h})=0  \ \ .$\\
\end{proof}

\subsection{Second cohomology $H^{2}(w_{\infty},w_{\infty}).$}
In this section, we aim compute the second cohomology group of $w_{\infty}$ with values in itself.
\begin{thm}
The second cohomology of the $3$-ary-Nambu-Lie superalgebra $w_{\infty})$ with values in the adjoint module vaniches, i.e.
$$H^{2}(w_{\infty},w_{\infty})=\{0\}.$$
Then, every formal deformation is equivalent to trivial deformation.
\end{thm}
\begin{proof}
For all $g\in Z^{1}_{0}(w_{\infty})$  (resp. $f\in Z^{2}_{0,s}(w_{\infty})$)  we have  respectively
\begin{eqnarray*}
&&\delta^{1}(g)(L_{p}^{v}\wedge L_{q}^{l},X_k^h)\\
&=&-g\Big(\big[L_{p}^{v}\wedge L_{q}^{l},X_k^h\big]\Big)
+\Big[L_p^v\wedge  L_{q}^{l},g(X_{k}^{h})\Big]
-\Big[X_k^h\wedge  L_{q}^{l},g(L_{p}^{v})\Big]
+\Big[X_k^h\wedge  L_{p}^{v},g(L_{q}^{l})\Big]=0.\\
\end{eqnarray*}
\begin{eqnarray*}
&&\delta^{2}(f)(L_{m}^{i}\wedge L_{n}^{j},  X_{p}^{v}\wedge X_{q}^{l} ,X_{k}^{h})\\
&=&-
f\Big( [L_{m}^{i}\wedge L_{n}^{j},  X_{p}^{v}]\wedge X_{q}^{l} ,X_{k}^{h}\Big)
-f\Big(X_{p}^{v},[L_{m}^{i}\wedge L_{n}^{j},X_{q}^{l}],X_{k}^{h} \Big)
-f\Big(X_{p}^{v}\wedge X_{q}^{l},[L_{m}^{i}\wedge L_{n}^{j},X_{k}^{h}]\Big)\\
&&+f\Big(L_{m}^{i}\wedge L_{n}^{j},[X_{p}^{v}\wedge X_{q}^{l},X_{k}^{h}]\Big)
+\Big[L_{m}^{i}\wedge L_{n}^{j},f( X_{p}^{v}\wedge X_{q}^{l},X_{k}^{h})\Big]
-\Big[X_{p}^{v}\wedge  X_{q}^{l},f( L_{m}^{i}\wedge L_{n}^{j},X_{k}^{h})\Big]\\
&&+\Big[X_{k}^{h}\wedge  X_{q}^{l},f( L_{m}^{i}\wedge L_{n}^{j},X_{p}^{v})\Big]
-\Big[X_{k}^{h}\wedge  X_{p}^{v},f( L_{m}^{i}\wedge L_{n}^{j},X_{q}^{l})\Big]=0.
\end{eqnarray*}

\textbf{Case 1}: $s\neq 0$.\\
Setting $m=0,\ n=0,\ i=0,\ j=1$, in (\ref{thomm1}), we obtain
\begin{eqnarray*}
&&\big(p+q+k\big) f\big( X_{p}^{v}\wedge X_{q}^{l} ,X_{k}^{h}\big)\\
&&+\big(h(q-p)+l (p-k)+v(k-q)          \big)
f\big(L_{0}^{0}\wedge L_{0}^{1},X_{p+q+k}^{v+l+h-1}\big)
+\big[L_{0}^{0}\wedge L_{0}^{1},f( X_{p}^{v}\wedge X_{q}^{l},X_{k}^{h})\big]\\
&&-\big[X_{p}^{v}\wedge  X_{q}^{l},f( L_{0}^{0}\wedge L_{0}^{1},X_{k}^{h})\big]
+\big[X_{k}^{h}\wedge  X_{q}^{l},f( L_{0}^{0}\wedge L_{0}^{1},X_{p}^{v})\big]
-\big[X_{k}^{h}\wedge  X_{p}^{v},f( L_{0}^{0}\wedge L_{0}^{1},X_{q}^{l})\big]
=0.
\end{eqnarray*}
 We denote by $g$ the linear map defined on $w_{\infty}$ by $g(X_{k}^{h})=-\frac{1}{s}f(L_{0}^{0}\wedge L_{0}^{1},X_{k}^{h})$.\\
It easy to verify that
$\delta^{1}(g)(X_{p}^{v}\wedge X_{q}^{l},X_k^h)=f(X_{p}^{v}\wedge X_{q}^{l},X_k^h).$\\
\textbf{Case 2}: $s= 0$.\\

Setting $m=1,\ n=-1,\ i=1,\ j=0$  in (\ref{thomm1}), we obtain
\begin{eqnarray*}
&&-2
                    f\big( X_{p}^{v}\wedge X_{q}^{l} ,X_{k}^{h}\big)
+\big(h(q-p)+l (p-k)+v(k-q)          \big)
f\big(L_{1}^{1}\wedge L_{-1}^{0},X_{p+q+k}^{v+l+h-1}\big)\\
&&-\big[X_{p}^{v}\wedge  X_{q}^{l},f( L_{1}^{1}\wedge L_{-1}^{0},X_{k}^{h})\big]
+\big[X_{k}^{h}\wedge  X_{q}^{l},f( L_{1}^{1}\wedge L_{-1}^{0},X_{p}^{v})\big]
-\big[X_{k}^{h}\wedge  X_{p}^{v},f( L_{1}^{1}\wedge L_{-1}^{0},X_{q}^{l})\big]
= 0
\end{eqnarray*}

 We denote by $g$ the linear map defined on $w_{\infty}$ by $g(X_{k}^{h})=-\frac{1}{2}f(L_{1}^{1}\wedge L_{-1}^{0},X_{k}^{h})$.\\
It easy to verify that
$\delta^{1}(g)(X_{p}^{v}\wedge X_{q}^{l},X_k^h)=f(X_{p}^{v}\wedge X_{q}^{l},X_k^h),$ that ends the proof.

\end{proof}

\end{document}